\documentclass[english,12pt]{article}
\usepackage{amsmath,amsfonts,amssymb,amsthm,fancyhdr,bm,verbatim,hyperref}
\usepackage{fullpage}
\usepackage{mathtools}
\usepackage{todonotes}
\usepackage{xcolor}
\usepackage{tkz-graph}
\usepackage{relsize}
\usepackage{comment}
\usetikzlibrary{positioning}

\makeatletter
\theoremstyle{plain}
\newtheorem{thm}{\protect\theoremname}
\theoremstyle{plain}

\theoremstyle{plain}
\newtheorem{lem}[thm]{\protect\lemmaname}
\newtheorem{conj}[thm]{\protect\conjecturename}
\newtheorem{prop}[thm]{\protect\propositionname}
\theoremstyle{definition}
\newtheorem*{Def}{Definition}

\newcommand{\E}{\mathbb{E}}
\newcommand{\Bin}{\text{Bin}}

\usepackage{babel}
\providecommand{\corollaryname}{Corollary}
\providecommand{\lemmaname}{Lemma}
\providecommand{\theoremname}{Theorem}
\providecommand{\conjecturename}{Conjecture}
\providecommand{\propositionname}{Proposition}

\begin{document}

\title{On the subgraph query problem}
\author{Ryan Alweiss\thanks{Department of Mathematics, Princeton University, Princeton, NJ 08541, USA. Email: {\tt alweiss@math.princeton.edu}. Research supported by a NSF Graduate Research Fellowship.} \hspace{5mm} Chady Ben Hamida \hspace{5mm} Xiaoyu He\thanks{Department of Mathematics, Stanford University, Stanford,
CA 94305, USA. Email: {\tt alkjash@stanford.edu}. Research supported by a NSF Graduate Research Fellowship.}\hspace{5mm} Alexander Moreira}
\maketitle

\begin{abstract}

Given a fixed graph $H$, a real number $p\in(0,1)$, and an infinite Erd\H{o}s-R\'enyi graph $G\sim G(\infty,p)$, how many adjacency queries do we have to make to find a copy of $H$ inside $G$ with probability at least $1/2$? Determining this number $f(H,p)$ is a variant of the {\it subgraph query problem} introduced by Ferber, Krivelevich, Sudakov, and Vieira. For every graph $H$, we improve the trivial upper bound of $f(H,p) = O(p^{-d})$, where $d$ is the degeneracy of $H$, by exhibiting an algorithm that finds a copy of $H$ in time $o(p^{-d})$ as $p$ goes to $0$. Furthermore, we prove that there are $2$-degenerate graphs which require $p^{-2+o(1)}$ queries, showing for the first time that there exist graphs $H$ for which $f(H,p)$ does not grow like a constant power of $p^{-1}$ as $p$ goes to $0$. Finally, we answer a question of Feige, Gamarnik, Neeman, R\'acz, and Tetali by showing that for any $\delta < 2$, there exists $\alpha < 2$ such that one cannot find a clique of order $\alpha \log_2 n$ in $G(n,1/2)$ in $n^\delta$ queries.
\end{abstract}

\section{Introduction}

The {\it subgraph query problem}, introduced by Ferber, Krivelevich, Sudakov, and Vieira \cite{fksv}, has been the subject of recent attention in extremal combinatorics and theoretical computer science. The problem is to determine the smallest number of adaptive queries of the form ``is $(u,v) \in E(G)$?" that need to be made to an Erd\H{o}s-R\'enyi random graph $G\sim G(n,p)$ to find a copy of a given subgraph $H$ with probability at least $\frac{1}{2}$. 

Several variants of the problem appear in the literature. Ferber, Krivelevich, Sudakov, and Vieira \cite{fksv, fksv2} first studied the subgraph query problem when $H$ is a long path or cycle of order comparable to $n$, exhibiting asymptotically optimal algorithms for finding long paths and cycles. For example, as long as $p\ge \frac{\log n + \log \log n +\omega(1)}{n}$ is above the threshold for Hamiltonicity in $G(n,p)$, they showed that a Hamiltonian cycle can be found by the time one reveals $(1+o(1))n$ edges. Here and henceforth we write $\log$ for the natural logarithm and $\lg$ for the base-$2$ logarithm.

In connection with the online Ramsey number, Conlon, Fox, Grinshpun, and He \cite{cfgh} studied the case where $H = K_m$ is a fixed complete graph, $p\rightarrow 0$, and the number of vertices $n$ is allowed to be arbitrarily large. They defined the function $f(H,p)$ to be the number of queries needed to find a copy of $H$ in the countably infinite random graph $G(\infty, p)$ with probability $\frac{1}{2}$, and proved that
\begin{equation}\label{eq:cfgh}
p^{-(2-\sqrt{2})m + O(1)} \le f(K_m, p) \le p^{-\frac{2}{3} m - O(1)}.
\end{equation}

In this paper, we study the behavior of $f(H,p)$ as $p\rightarrow 0$ for an arbitrary fixed graph $H$. We will use the phrases ``build $H$ in $T$ time" and ``find $H$ in $T$ queries" interchangeably for the statement $f(H,p) \le T$.

Recall that a graph $H$ is {\it$d$-degenerate} if every subgraph of $H$ contains a vertex of degree at most $d$, and the {\it degeneracy} of $H$ is the least $d$ for which $H$ is $d$-degenerate. Equivalently, $H$ is $d$-degenerate if and only if there is an acyclic orientation of $H$ with maximum outdegree at most $d$. Degeneracy is the natural notion of sparsity in graph Ramsey theory (see e.g. the recent proof of the Burr-Erd\H{o}s conjecture by Lee \cite{cl}). 

In the subgraph query problem, a $d$-degenerate graph can be built by adding one vertex at a time so that each new vertex has degree at most $d$ at the time it is built. Since a common neighbor of $d$ given vertices can be found in $O(p^{-d})$ queries, this shows that $f(H,p) = O_H(p^{-d})$ whenever $H$ is $d$-degenerate.

Our first main result is that this trivial bound is never tight when $d\ge 2$. Define the {\it depth} $\Delta$ of a graph $H$ with degeneracy $d$ to be the smallest $\Delta$ for which there exists an acyclic orientation of $H$ with maximum outdegree at most $d$ and longest directed path of length at most $\Delta$ (we use the convention that the length of the path with $n+1$ vertices is $n$). Let $\log_t(x)$ denote the $t$-fold iterated logarithm of $x$.

\begin{thm}\label{thm:upper-general}
If $H$ is a graph with degeneracy $d\ge 2$ and depth $\Delta\ge 1$, then
\[
f(H, p) = O_H\Big(\frac{p^{-d} \log_{\Delta + 1}(p^{-1})}{\log_{\Delta} (p^{-1})}\Big).
\]
\end{thm}

Roughly speaking, one of the main innovations is to exploit the observation that in a random graph $G(n,1/n)$, the degrees of vertices are approximately Poisson with mean $1$. Thus, the maximum degree is $\Omega(\log n/\log \log n)$ despite the fact that the average degree is constant. Repeatedly finding these vertices of exceptionally large degree allows us to find $H$ slightly faster.

We will also show that the behavior in Theorem~\ref{thm:upper-general} can be correct up to the polylogarithmic factor. Let the triforce graph be the graph obtained from the triangle $K_3$ by adding a common neighbor to each pair of vertices (see Figure~\ref{fig:triforce} in Section~\ref{sec:upper}).

\begin{thm}\label{thm:lower}
If $H$ is the triforce graph and $\ell = \frac{\log(1/p)}{2\log \log(1/p)}$, then
\[
f(H,p) = \Omega(p^{-2}/\ell^4).
\]
\end{thm}

Note that the triforce is $2$-degenerate, so Theorems~\ref{thm:upper-general} and~\ref{thm:lower} together prove that $f(H,p) = o(p^{-2})$ and $f(H,p) = \Omega(p^{-2+\varepsilon})$ for every $\varepsilon > 0$. This is the first example of a graph for which it is known that $f(H,p)$ does not grow like a power of $p^{-1}$.

The question of querying for subgraphs in random graphs was also studied by Feige, Gamarnik, Neeman, R\'acz, and Tetali \cite{fgnrt}, and by R\'acz and Schiffer in the related planted clique model \cite{rs}. Feige, Gamarnik, Neeman, R\'acz, and Tetali restricted their attention to the balanced random graph $G(n,\frac{1}{2})$ and asked for the minimum number of queries needed to find a clique of order $(2-o(1))\lg n$ (which approaches the clique number) with probability at least $1/2$. For $\delta < 2$, define $\alpha_\star(\delta)$ to be the supremum over $\alpha \le 2$ for which a clique of order $\alpha \lg n$ can be found with probability at least $1/2$ in at most $n^\delta$ queries for all $n$ sufficiently large. They showed under the additional assumption that only a bounded number of rounds of adaptiveness are used that $\alpha_\star(\delta) < 2$ for all $\delta < 2$, and asked if this could be proved unconditionally.

Our last theorem answers this question affirmatively. We are grateful to Huy Pham~\cite{p} for communicating to us the main idea of the proof.

\begin{thm}\label{thm:cliques}
For all $2/3<\delta<2$,
\[
\alpha_\star(\delta)\le1+\sqrt{1-\frac{(2-\delta)^{2}}{2}}<2.
\]
\end{thm}

The proof is an adaptation of the lower bound proof for (\ref{eq:cfgh}) by \cite{cfgh} to take the size of the vertex set into account. The exact value of $\alpha_\star(\delta)$ remains open for all $\delta$, and the best known lower bound is $\alpha_\star(\delta) \ge 1 + \frac{\delta}{2}$ when $1\le \delta < 2$ (see \cite[Lemma 6]{fgnrt}).

\vspace{3mm}

\noindent {\bf Organization.} In Section~\ref{sec:upper}, we describe a new algorithm for finding any $d$-degenerate graph and prove that it achieves the runtime described in Theorem~\ref{thm:upper-general}. In Section~\ref{sec:lower}, we give a new argument for proving lower bounds on $f(H,p)$, proving Theorem~\ref{thm:lower}. In Section~\ref{sec:clique} we give a short proof of Theorem~\ref{thm:cliques}, using a variation of the methods in \cite{cfgh}. Finally, Section~\ref{sec:closing} highlights a few of the many open questions that remain about $f(H,p)$.

We will write $b=p^{-1}$ for the expected number of queries needed to find a single edge in $G(\infty, p)$. No attempt will be made to optimise the implicit constants in any of our proofs. We use $A \lesssim B$ to mean $A=O(B)$.  For the sake of clarity of presentation, we systematically omit floor and ceiling signs whenever they are not crucial.

\section{Upper bounds}\label{sec:upper}

\subsection{An Illustrative Example}

 As mentioned in the introduction, there is a straightforward algorithm for finding any $d$-degenerate graph $H$ in $O_H(b^d)$ time. In this section we prove Theorem~\ref{thm:upper-general} by providing a new algorithm that beats the trivial algorithm by an iterated logarithmic factor.

We begin by illustrating how the algorithm works with a specific $2$-degenerate graph.

\begin{Def}

The \emph{triforce} is the graph on $6$ vertices and $9$ edges obtained from the triangle $K_3$ by adding a common neighbor to each pair of vertices.
\end{Def}

\begin{figure}[h]\label{fig:triforce}
\centering
\begin{tikzpicture} [rotate = -90]
\GraphInit[vstyle=Classic]
\SetVertexNoLabel
\SetGraphUnit{1}
\tikzset{VertexStyle/.append style={fill=black, minimum size=2pt, inner sep=0pt}}
\Vertex{a}
\SOEA(a){b}
\NOEA(a){c}
\SOEA(b){d}
\SOEA(c){e}
\NOEA(c){f}
\Edges(a,b,c,a)  
\Edges(e,d,b,e,c,f,e)
\end{tikzpicture}
\caption{The triforce graph.}
\end{figure}
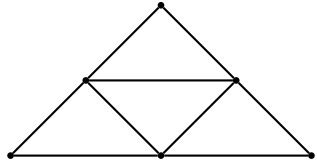

The key step in building the triforce quickly is to build a large book.

\begin{Def}
The {\it book} $B_{d,t}$ is the graph on $d+t$ vertices obtained by removing the edges of a clique $K_t$ from a complete graph $K_{d+t}$. The $t$ vertices of the removed clique are called the {\it pages} of the book and the remaining $d$ vertices are called its {\it spine}.
\end{Def}
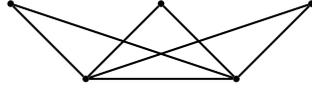
\begin{figure}[h]\label{fig:book}
\centering
\begin{tikzpicture} [rotate = -0]
\GraphInit[vstyle=Classic]
\SetVertexNoLabel
\SetGraphUnit{1}
\tikzset{VertexStyle/.append style={fill=black, minimum size=2pt, inner sep=0pt}}
\Vertex{a}
\NOEA(a){2}
\SOEA(2){b}
\NOWE(a){1}
\NOEA(b){3}
\Edges(1,a,b,1)
\Edges(2,b,3)
\Edges(2,a,3)
\end{tikzpicture}
\caption{The book $B_{2,3}$.}
\end{figure}
Note that when $d$ and $t$ are fixed positive integers, $B_{d,t}$ is $d$-degenerate and thus we have $f(B_{d,t}, p) = O_t(b^d)$. The key observation is that this can be improved substantially even if we allow $t$ to grow slowly as $p$ tends to $0$.

\begin{lem}\label{lem:book}
If $d\ge 2$ and $\ell=\frac{\log b}{2\log \log b}$, then $f(B_{d,\ell},p)=O(b^d \ell^{-1/2})$.
\end{lem}
\begin{proof}
We will exhibit an algorithm which finds $B_{d,\ell}$ in $G(\infty, p)$ with constant probability (w.c.p.) in $O(b^d\ell^{-1/2})$ time. The algorithm has three steps.

First, we find w.c.p. $d-1$ vertices $v_1,\ldots, v_{d-1}$ of the spine forming a clique in $O(b^{d-2})$ time, which is possible because $K_{d-1}$ is $(d-2)$-degenerate. Assume this step succeeds.

Next, we build a large pool $S$ of common neighbors of $v_1,\ldots, v_{d-1}$, which will serve as candidates for the remaining vertex $v_d$ of the spine and for the pages of the book. In $d-1=O(1)$ queries we can check a single new vertex $u$ to see if it is a common neighbor of $v_1,\ldots, v_{d-1}$, and $u$ has a probability $p^{d-1}$ of being such a common neighbor. We check a total of $4b^d \ell^{-1/2}$ possible $u$, and each common neighbor successfully found is added to $S$. Since the outcomes of all queries are independent, $|S|$ is distributed like the binomial random variable $\text{Bin}(4b^d\ell^{-1/2}, p^{d-1})$ with mean $4b\ell^{-1/2}$, so w.c.p. $|S| \ge 2b\ell^{-1/2}$.

For the last step, assuming the previous two steps succeed, we will find a star $K_{1, \ell}$ contained in $S$. Along with vertices $v_1,\ldots ,v_{d-1}$ already chosen, this forms the desired book.

To find this star, remove vertices from $S$ until it has size exactly $2b\ell^{-1/2}$, and then query all pairs of vertices in $S$ in $O(b^2 \ell^{-1})$ time. The induced subgraph on $S$ is just an Erd\H{o}s-R\'enyi random graph $G(2b\ell^{-1/2}, p)$. It suffices to show that w.c.p. there exists a vertex of degree at least $\ell$ therein. This fact is a consequence of the observation that the degrees are approximately Poisson.

To give a quick proof of this fact, divide $S$ into two sets $S_1, S_2$ of size $r = b\ell^{-1/2}$, let $u_1$,\ldots, $u_{r}$ be the vertices of $S_1$, and let $X_i$ be the number of neighbors of $u_i$ in $S_2$. Then $\{X_i\}_{i=1}^{r}$ are $r$ i.i.d. random variables distributed like $\text{Bin}(r, p)$, so 
\[
\mathbb{P}[X_i \ge \ell] \ge \binom{r}{\ell} p^{\ell}(1-p)^{r-\ell} \ge \frac{(r - \ell)^\ell p^\ell (1-p)^{r}}{\ell!}.
\]
As $p\rightarrow 0$, we can bound $r - \ell > b\ell^{-1/2}/2$, $(1-p)^{r} \rightarrow 1$, and $\ell! < \ell^\ell$. Thus,
\[
\mathbb{P}[X_i \ge \ell] \ge \Omega\Big(\frac{1}{2^\ell \ell^{3\ell/2}}\Big).
\]
When $\ell = \frac{\log b}{2\log \log b}$, this fraction is certainly $\Omega(b^{-\frac{4}{5}})$. In particular, since there are $r = b^{1-o(1)}$ independent random variables $X_i$, w.c.p. some $X_i$ is at least $\ell$, as desired.

Letting the vertex of degree $\ell$ be the last vertex $v_d$ of the book's spine and its $\ell$ neighbors in $S$ be the pages of the book, we have found a copy of $B_{d,\ell}$ w.c.p. in $O(b^d)$ total queries, as desired.
\end{proof}

We are now ready to prove a stronger version of Theorem~\ref{thm:upper-general} when $H$ is the triforce.

\begin{thm}\label{thm:upper-triforce}
If $H$ is the triforce graph and $\ell=\frac{\log b}{2\log \log b}$, then $f(H,p) = O(b^2\ell^{-\frac{1}{2}})$.
\end{thm}
\begin{proof}
We exhibit an algorithm for finding $H$ w.c.p. in $O(b^2\ell^{-\frac{1}{2}})$ time.

Using Lemma~\ref{lem:book} with $d=2$, build a copy of $B_{2,\ell }$.  Let $x$ and $y$ be the two vertices of its spine and let $Z$ be its pages.  In $O(b^2\ell^{-\frac{1}{2}})$ time we can w.c.p. find two sets of vertices $S_x, S_y$, each of size $b\ell^{-\frac{1}{2}}$, so that everything in $S_x$ is adjacent to $x$ and everything in $S_y$ is adjacent to $y$. Now, we will query all pairs between $S_x$ and $Z$ as well as all pairs between $S_y$ and $Z$.  This takes only $O(b\ell^{1/2})$ time, which is negligible.

We claim that w.c.p. there exist $x' \in S_x$, $z \in Z$, and $y' \in S_y$ so that $x'\sim z$ and $z\sim y'$. This follows because w.c.p., $\Theta(\ell^{1/2})$ vertices of $Z$ have a neighbor in $S_x$, and among these vertices w.c.p. at least one has a neighbor in $S_y$. Now, let $z'$ be any common neighbor of $x$ and $y$ in $Z$ other than $z$. It follows that the six vertices $x,y,z,x',y',z'$ form a triforce, and we have found it in $O(b^2\ell^{-\frac{1}{2}})$ queries w.c.p., as desired.
\end{proof}

\subsection{The general upper bound}

Roughly speaking, the main trick in the proofs of Lemma~\ref{lem:book} and Theorem~\ref{thm:upper-triforce} is that we can find vertices of much larger than average degree in a random graph with constant average degree. We will iterate this trick many times to prove the general statement in Theorem~\ref{thm:upper-general}.  

We will construct an arbitrary $d$-degenerate graph $H$ recursively. If the vertices of $H$ are ordered $v_1,\ldots, v_n$ in the degeneracy order, the algorithm will maintain a ``cloud" of candidates $C_i$ for the image of vertex $v_i$ which shrinks as the algorithm progresses. On step $i$, the algorithm chooses $v_i$ from $C_i$ and then shrinks the clouds corresponding to neighbors of $v_i$ to stay consistent with this choice.

\begin{proof}[Proof of Theorem~\ref{thm:upper-general}.]

Let $H$ be a graph on $n$ vertices with degeneracy $d\ge 2$ and depth $\Delta$. Order its vertices $v_1,\ldots, v_n$ so that each $v_i$ has at most $d$ neighbors $v_j$ with $j<i$, and the longest left-to-right path $v_{i_0},\ldots, v_{i_r}$ with $i_0<\cdots < i_r$ has length $r=\Delta$. Let $\Delta_i$ be the length (in edges) of the longest left-to-right-path ending at $v_i$, so that $\Delta_i \le \Delta$ for all $i$. Finally, define 
\[
L(x) = \frac{\log x}{3n \log \log x}
\]
and $\ell_i = L^{(\Delta_i)}(b)$ is obtained from $b$ by iterating $L$ $\Delta_i$ times.

We describe an algorithm for finding an injection $\phi$ from $H$ to $G(\infty, p)$ in a series of rounds, assuming $p$ is sufficiently small. There are many points at which the algorithm may fail. However, each round succeeds with probability $\Omega_H(1)$ conditional on the success of all previous rounds, and there are $n = O_H(1)$ rounds, so the entire algorithm succeeds with $\Omega_H(1)$ probability. The algorithm can then be repeated a number of times depending only on $H$ until its success probability reaches $\frac{1}{2}$; this only changes the implicit constant in $f(H,p)$.

We begin by setting aside $n$ disjoint sets (``clouds") $C_1, \ldots, C_n$ which will change throughout the algorithm.  We initialize these to $C_1^{(0)},\ldots, C_n^{(0)}$ of order $|C_i^{(0)}|=b^d/\ell_i$, so that $C_i^{(0)}$ is the set of candidates for $\phi(v_i)$. We proceed in $n$ rounds, so that $C_j^{(k)}$ will refer to the state of cloud $C_j$ after round $k$.  After the $k$th round we will have nonempty disjoint sets $C_j^{(k)}$, and we always have $C_j^{(k-1)} \supset C_j^{(k)}$. In the round $k$ a number of queries are made to decide the value of $\phi(v_k)\in C_k^{(k-1)}$. Thus $C_k^{(k)}$ is the singleton $\{\phi(v_k)\}$ and $C_k$ remains a singleton until the end. For each $j$ with $j>k$ and $v_j\sim v_k$, the set $C_j^{(k-1)}$ is updated to a subset $C_j^{(k)}$ consisting of all elements of $C_j^{(k-1)}$ adjacent to $\phi(v_k)$. We say that a vertex $v_j$ is {\it living} after round $k$ if $j > k$, and {\it dead} otherwise.

Two properties are maintained. The first is that after round $k$, for any $i\le k$ and $j\le n$ and any $u_i \in C_i^{(k)}$ and $u_j \in C_j^{(k)}$, $u_i\sim u_j$ if $v_i \sim v_j$. In other words, the adjacency relations are correct within the dead vertices and between the dead vertices and the clouds $C_i^{(k)}$ for the living ones.

The second property is that the size of the set $C_j^{(k)}$ must be
\[
c_j^{(k)} \coloneqq \begin{cases}
    \frac{b^{d-m}}{\ell_j} & \text{if $v_j$ is living and has $m < d$ dead left-neighbors,} \\
    \ell_j & \text{if $v_j$ is living and has exactly $d$ dead left-neighbors,} \\
    1 & \text{if $v_j$ is dead.}
    \end{cases}
\]

The queries on round $k$ are made to guarantee two properties. First, on round $k$, vertices are thrown out of $C_k^{(k-1)}$ until it has exactly $\ell_k$ vertices (this is possible because $c_k^{(k-1)} \ge \ell_k$ when $b$ is sufficiently large). Then, consider the $j$ so that $j>k$ and $v_j \sim v_k$. If such a $v_j$ has exactly $d-1$ dead left-neighbors, then $j$ is called {\it active} on round $k$, and otherwise $j$ is called {\it inactive}. For each active $j$, all pairs in $C_k^{(k-1)} \times C_j^{(k-1)}$ are queried.

Each round is divided into an {\it active portion}, which happens first, and then an {\it inactive portion}.

The active portion of round $k$ succeeds if a candidate $u_k \in C_k^{(k-1)}$ is found to have at least $c_j^{(k)}$ neighbors in $C_j^{(k-1)}$ for all the active $j$. One such candidate $u_k$ is picked for $\phi(v_k)$ and $C_j^{(k)}$ is chosen to be exactly $c_j^{(k)}$ neighbors of $u_k$ in $C_j^{(k-1)}$. 

Then, for all inactive $j$, all pairs $\{u_k\} \times C_j^{(k-1)}$ are queried. The inactive portion of round $k$ succeeds if after these queries a total of $c_j^{(k)}$ neighbors are found for $u_k$ in $C_j^{(k-1)}$ for each of the inactive $j$'s as well. We say that the round succeeds if both the active and inactive portions succeed. The algorithm only continues past round $k$ if round $k$ succeeds.

By induction on $k$, the algorithm maintains all the required properties and produces a valid injection $\phi:H\rightarrow G(\infty, p)$ if it succeeds on every round. It remains to show that the probability of success on each round is $\Omega_H (1)$.

For each $u \in C_k^{(k-1)}$ and $j>k$ for which $v_j \sim v_k$, let $d_j (u)$ be number of neighbors $u$ has in $C_j^{(k-1)}$. Note that $d_j(u)$ is distributed like $\text{Bin}(c_j^{(k-1)}, p)$, since each vertex of $C_j^{(k-1)}$ is adjacent to $u$ independently with probability $p$.

Suppose $j$ is active in round $k$, so that $c_j^{(k-1)} = b/\ell_j$. This time, we get
\[
\mathbb{P}[d_j(u) \ge c_j^{(k)}] = \mathbb{P}[\Bin(b/\ell_j, p) \ge \ell_j] \ge \binom{b/\ell_j}{\ell_j}p^{\ell_j}(1-p)^{b/\ell_j - \ell_j}.
\]
Using the facts that $1-x \ge e^{-2x}$ for all $x \in [0, \frac{1}{2}]$ and that $\ell_j\rightarrow \infty$ as $p\rightarrow 0^+$, we see that $(1-p)^{b/\ell_j - \ell_j} \ge e^{-2/\ell_j} \rightarrow 1$. Also, $\binom{a}{b} \ge (a/b)^b$ for all $a\ge b\ge 1$, and so
\[
\mathbb{P}[d_j(u) \ge c_j^{(k)}] \ge \binom{b/\ell_j}{\ell_j}p^{\ell_j}(1-p)^{b/\ell_j - \ell_j} = \Omega(\ell_j^{-2\ell_j}).
\]

There are at most $n$ total $j$, so taking a product over all active $j$, we arrive at a lower bound
\[
\mathbb{P}[d_j(u) \ge c_j^{(k)}\text{for all active $j$}] \ge \Omega_H (\ell_j^{-2n\ell_j}).
\]
Since each $u\in C_k^{(k-1)}$ is individually a successful candidate for $\phi(v_k)$ with this probability, and these are $\ell_k$ independent events, it follows that
\[
\mathbb{P}[\text{Active portion round $k$ succeeds}] \ge \Omega_H (\min(1, \ell_k \ell_j^{-2n \ell_j})).
\]
Finally, we observe that for every $j$ active in round $k$, $\Delta_j \ge \Delta_k + 1$ since every left-to-right path ending at $k$ extends to a longer one ending at $j$. Thus, $\ell_j \le L(\ell_k)$, and the function $L$ was chosen so that $L(x)^{2n L(x)} \le x$ for $x$ sufficiently large. It follows that the active portion of round $k$ succeeds with probability $\Omega_H(1)$, as desired.

Now we look at the inactive $j$ in round $k$. Then $c_j^{(k)} = p c_j^{(k-1)}$, so
\[
\mathbb{P}[d_j(u_k) \ge c_j^{(k)}] = \mathbb{P}[\Bin(c_j^{(k-1)}, p) \ge p c_j^{(k-1)}] = \Omega(1).
\]
Thus, conditional on the success of the active portion of round $k$, the inactive portion succeeds with probability $\Omega_H(1)$ as well.

We have now shown that the algorithm, iterated $O_H(1)$ times, succeeds with probability $\frac{1}{2}$. It remains to bound the total number of queries made. In the active portion of each round, queries are only made between sets $C_k^{(k-1)}$ and $C_j^{(k-1)}$ if $j$ is relevant, which implies that $c_j^{(k)} = b/\ell_j$. Also, elements of $C_k^{(k-1)}$ were thrown out until it had size exactly $\ell_k = O(b)$, so the number of queries made in the active portion of any round is at most $O(b^2/L^{(\Delta)}(b)) = O(b^d/L^{(\Delta)}(b))$.

In the inactive portion of each round, queries are made between a single vertex $u_k$ and sets $C_j^{(k-1)}$ of size at most $b^d/\ell_j$ each. Thus, the number of queries made in the inactive portion of any round is also $O(b^d/L^{(\Delta)}(b))$.

Since there are at most $n = O_H(1)$ rounds and at most $n$ choices of $j$ involved in each round, we find that
\[
f(H, p) = O_H\Big(\frac{b^d}{L^{(\Delta)}(b)}\Big) = O_H\Big(\frac{p^{-d} \log_{\Delta + 1}(p^{-1})}{\log_{\Delta} (p^{-1})}\Big).
\]
as desired.
\end{proof}

\section{Lower bounds}\label{sec:lower}

\subsection{Preliminaries}

In this section, we will prove lower bounds for $f(H,p)$.  Because $N$ queries necessarily involve at most $2N$ vertices, it suffices to prove lower bounds for finding a copy of $H$ in $G(2N,p)$ rather than in $G(\infty,p).$  Following \cite{cfgh}, we will lower bound the number of queries it takes to build a copy of $H$ by showing that the expected number of copies of $H$ we can build in some amount of time is not too large.

\begin{Def}
If $H$ is a graph without isolated vertices, define $t(H,p,N)$ to be the maximum (over all querying strategies) expected number of copies of $H$ that can be found in $G(\infty,p)$ in $N$ queries. Since we are working on $G(2N,p)$, if $H = H'\cup\{v_1,\ldots, v_t\}$ has $t$ isolated vertices, define $t(H,p,N)\coloneqq (2N)^t \cdot t(H',p,N)$.
\end{Def}

If we show that we cannot build so many copies of $H$ (in expectation) in some time, this gives us a lower bound on how long it takes to build a single copy of $H$.

\begin{lem}\label{lemma:LB1}
If $N\ge f(H,p)$, then
\[
f(H,p) \cdot t(H,p,N) \ge N/4.
\]
\end{lem}

Thus, upper bounds on $t(H,p,N)$ will yield lower bounds on $f(H,p)$.  The proof of Lemma~\ref{lemma:LB1} is straightforward, but we include it for completeness.

\begin{proof}
By definition, there exists a strategy which finds $H$ with probability $\frac{1}{2}$ in $f(H,p)$ queries. Given $N$ queries, we can iterate this strategy $\lfloor N/f(H,p)\rfloor$ independent times on disjoint vertex sets. By linearity of expectation, this implies 
\[
t(H,p,N) \ge \Big\lfloor\frac{N}{f(H,p)}\Big\rfloor \cdot \frac{1}{2} \ge \frac{N}{4f(H,p)}.\qedhere
\]
\end{proof}

Thus, it suffices to produce upper bounds on $t(H,p,N)$. Fortunately, we can recursively bound $t(H,p,N)$ in terms of $t(H',p,N)$ for some subgraphs $H'$.  The following bounds are proved in \cite{cfgh}.

\begin{lem}[\cite{cfgh}]\label{lemma:LB2}
If $H$ is any graph, $p\in (0,1)$, and $N \ge p^{-1 -\varepsilon}$ for some $\varepsilon > 0$, then the following inequalities hold:
\begin{align}
t(H,p,N) &\le \min\limits_{e \in E(H)} t(H\backslash e,p,N) \label{eq:t-rec-1} \\
t(H,p,N) &\lesssim p \max\limits_{e\in E(H)} t(H\backslash e, p, N) \label{eq:t-rec-2} \\
t(H,p,N) &\lesssim pN \min\limits_{e\in E(H)} t(H\backslash\{u,v\}, p, N), \label{eq:t-rec-3}
\end{align}
where $u,v$ are the vertices of $e$ in (\ref{eq:t-rec-3}). In the latter two inequalities, the implicit constants are allowed to depend only on $H$.
\end{lem}

In general, the bounds of Lemma~\ref{lemma:LB2} are not tight. In certain cases, we will improve these bounds using the crucial observation that large enough sets of vertices in a random graph have few common neighbors. For any vertex subset $U\subseteq V(G)$ of a graph $G$, write $d(U)$ for the number of common neighbors of every vertex in $U$.

\begin{lem}
\label{lem:poisson}
Let $\ell=\frac{\log b}{2\log \log b}$, let $k, n\ge 2$ be absolute constants, and let $G=G(2N,p)$. 
\begin{enumerate}
\item
If $p^kN \lesssim 1$, then there exists $C>0$ so that 
\[
\mathbb{P}[\max_{|U|=k} d(U)>C\ell]<p^n, 
\]
where the maximum is taken over all $k$-subsets $U$ of $V(G)$. 
\item
If $p^kN=(1/N)^{\Omega(1)}$, then there exists $C>0$ so that 
\[
\mathbb{P}[\max_{|U|=k} d(U)>C]<p^n.
\]
\end{enumerate}
\end{lem}

\begin{proof}
For an arbitrary set of $k$ vertices $U$, note that $d(U)\sim \text{Bin}(2N-k,p^{k})$ as there are $2N-k$ other vertices of $G(2N,p)$ and each vertex has a $p^{k}$ chance of being adjacent to all members of $U$. Hence, we find that
\begin{equation}\label{eq:codegree}
\mathbb{P}[d(U) \ge t] =  \mathbb{P}[\text{Bin}(2N-k,p^{k}) \ge t] \leq \binom{2N-k}{t} p^{tk} \leq \Big(\frac{2Ne}{t}p^{k}\Big)^t.
\end{equation}

If $p^{k}N \lesssim 1$, then (\ref{eq:codegree}) implies $\mathbb{P}[d(U) \ge t] \le (O(1/t))^t$. Next we take a union bound over at most $(2N)^k$ choices of $U$, which shows that, if we take $t=C\ell$ for a sufficiently large $C$ depending on $n$,
\[
\mathbb{P}[\max_{|U|=k}d(U) \ge t] \le (O(1/t))^t \cdot N^k<p^n.
\]
This proves the first part of the lemma.

If $p^kN=(1/N)^{\Omega(1)}$, then (\ref{eq:codegree}) implies the stronger bound $\mathbb{P}[d(U) \ge t]=N^{-\Omega(t)}$. For any fixed $n$, if $C$ is a large enough constant, the probability that $\max_{|U|=k}d(U)>C$ will be below $N^{-\Omega(C)}N^k \le p^n$ by the union bound. \end{proof}

The power of Lemma~\ref{lem:poisson} is that the final graph that we find after $N$ queries is a subgraph of $G(2N,p)$, so we can bound the number of common numbers of any vertex set $U$ of constant size without even seeing the graph. It allows us to prove new upper bounds on $t(H,p,b^d)$.

\begin{lem}
\label{lem:reduction}
Let $\ell=\frac{\log b}{2\log \log b}$, let $H$ be a graph, let $v\in V(H)$, and let $H' = H \setminus \{v\}$.
\begin{enumerate}
\item
If $d(v) = d$, then 
\[
t(H,p,b^d) \lesssim \ell t(H',p,b^d). 
\]

\item
If $d(v) > d$, then 
\[
t(H,p,b^d) \lesssim t(H',p,b^d).
\]
\end{enumerate}
\end{lem}

\begin{proof}
Let $k= d(v)$, and let the neighbors of $v$ in $H$ be $v_1, \ldots, v_k$. Fix a query strategy that maximizes $t(H,p,b^d)$. 

For any subset $U=\{u_1, \ldots, u_k\}$ of $k$ vertices of the final graph $G \subset G(2b^d,p)$ that is found, let $H'(U)$ be the number of maps $H'\rightarrow G$ so that for all $1 \le i \le k$, $v_i$ maps to $u_i$. Then we have that
\begin{equation}\label{eq:t-H-expectation}
t(H,p,b^d)=\mathbb{E}\left[\sum_{|U|=k}d(U)H'(U)\right]\le \mathbb{E}\left[\left(\max_{|U|=k}d(U)\right) \left( \sum_{|U|=k} H'(U) \right)\right],
\end{equation}
where the sum is taken over all $k$-sets of vertices $U$.

Assume $d(v)=d$. By the first part of Lemma~\ref{lem:poisson}, there is a large constant $C=C(H)$ so that 
\[
\mathbb{P}\left[\max_{|U|=k}d(U)>C\ell\right]<p^{d|V(H')|+1}.
\]

We will break up the expectation in (\ref{eq:t-H-expectation}) depending on the size of $\max_{|U|=k} d(U)$. If $\max_{|U|=k} d(U) \le C\ell$, the contribution to right side of (\ref{eq:t-H-expectation}) is $O(\ell t(H',p,b^d))$.  Now, it holds that $\max_{|U|=k} d(U)>C\ell$ with probability at most $p^{d|V(H')|+1}$, so the contribution from these terms is bounded by $p^{d|V(H')|+1}(2b^d)^{|V(H')|}=o(1)$.  

Likewise, when $d(v)>d$, by the second part of Lemma~\ref{lem:poisson} there is a $C$ so that the case of $\max_{U} d(U) \leq C$ contributes $O(t(H',p,b^d))$ to the right side of (\ref{eq:t-H-expectation}), and the case of $\max_{|U|=k} d(U)>C$ contributes $o(1).$\end{proof}

\subsection{Proof of Theorem~\ref{thm:lower}}

We now begin the proof of Theorem~\ref{thm:lower}. The main idea is to apply Lemma~\ref{lem:reduction} to obtain new upper bounds on $t(H,p,N)$ for various subgraphs $H$ of the triforce, and then combine these with Lemma~\ref{lemma:LB2} to bound $t(H,p,N)$ for the triforce itself. 

We describe the subgraphs of the triforce to which we will apply Lemma~\ref{lem:reduction}. Any copy of the triforce must arise from a copy of one of the graphs formed by deleting two edges from the triforce. There are $8$ such graphs up to isomorphism, which we denote by $H_i$ for $1 \le i \le 8$ (see Figure~\ref{fig:H}).

\begin{figure}[h]
    \centering

\begin{tikzpicture} [rotate = -90]
\GraphInit[vstyle=Classic]
\SetVertexNoLabel
\SetGraphUnit{0.75}
\tikzset{VertexStyle/.append style={fill=black, minimum size=2pt, inner sep=0pt}}
\Vertex{a}
\SOEA(a){b}
\NOEA(a){c}
\SOEA(b){d}
\begin{scope}
\tikzset{VertexStyle/.append style={label={[label distance = 0.1cm]below:{$H_1$}}}}
\SOEA(c){e}
\end{scope}
\NOEA(c){f}
\NO(f){o}
\NOWE(o){n}
\NOWE(n){l}
\NOEA(l){m}
\begin{scope}
\tikzset{VertexStyle/.append style={label={[label distance = 0.1cm]below:{$H_2$}}}}
\SOEA(m){p}
\end{scope}
\NOEA(m){q}

\Edges(l,m,n)
\Edges(p,q,m,p,n)
\Edges(p,o)

\Edges(b,c,a)  
\Edges(d,b,e,c,f,e)

\NO(q){d1}
\NOWE(d1){b1}
\NOWE(b1){a1}
\NOEA(a1){c1}
\begin{scope}
\tikzset{VertexStyle/.append style={label={[label distance = 0.1cm]below:{$H_3$}}}}
\SOEA(c1){e1}
\end{scope}
\NOEA(c1){f1}
\NO(f1){o1}
\NOWE(o1){n1}
\NOWE(n1){l1}
\NOEA(l1){m1}
\begin{scope}
\tikzset{VertexStyle/.append style={label={[label distance = 0.1cm]below:{$H_4$}}}}
\SOEA(m1){p1}
\end{scope}
\NOEA(m1){q1}

\Edges(l1,m1,n1)
\Edges(p1,q1,m1,p1) 
\Edges(n1,o1,p1)

\Edges(c1,a1)  
\Edges(e1,d1,b1,e1,c1,f1,e1)

\end{tikzpicture}

\centering
\begin{tikzpicture} [rotate = -90]
\GraphInit[vstyle=Classic]
\SetVertexNoLabel
\SetGraphUnit{0.75}
\tikzset{VertexStyle/.append style={fill=black, minimum size=2pt, inner sep=0pt}}
\Vertex{a}
\SOEA(a){b}
\NOEA(a){c}
\SOEA(b){d}
\begin{scope}
\tikzset{VertexStyle/.append style={label={[label distance = 0.1cm]below:{$H_5$}}}}
\SOEA(c){e}
\end{scope}
\NOEA(c){f}
\NO(f){o}
\NOWE(o){n}
\NOWE(n){l}
\NOEA(l){m}
\begin{scope}
\tikzset{VertexStyle/.append style={label={[label distance = 0.1cm]below:{$H_6$}}}}
\SOEA(m){p}
\end{scope}
\NOEA(m){q}

\Edges(l,m)
\Edges(n,l)
\Edges(p,q,m,p)
\Edges(n,o,p)

\Edges(b,c,a)  
\Edges(e,d,b,e) 
\Edges(c,f,e)

\NO(q){d1}
\NOWE(d1){b1}
\NOWE(b1){a1}
\NOEA(a1){c1}
\begin{scope}
\tikzset{VertexStyle/.append style={label={[label distance = 0.1cm]below:{$H_7$}}}}
\SOEA(c1){e1}
\end{scope}
\NOEA(c1){f1}
\NO(f1){o1}
\NOWE(o1){n1}
\NOWE(n1){l1}
\NOEA(l1){m1}
\begin{scope}
\tikzset{VertexStyle/.append style={label={[label distance = 0.1cm]below:{$H_8$}}}}
\SOEA(m1){p1}
\end{scope}
\NOEA(m1){q1}

\Edges(m1,n1)
\Edges(p1,q1,m1,p1,n1,o1,p1)

\Edges(a1,b1,c1)  
\Edges(d1,b1,e1,c1,f1,e1)

\end{tikzpicture}

\caption{The eight subgraphs (up to isomorphism) with $7$ edges of the triforce.}
\label{fig:H}
\end{figure}
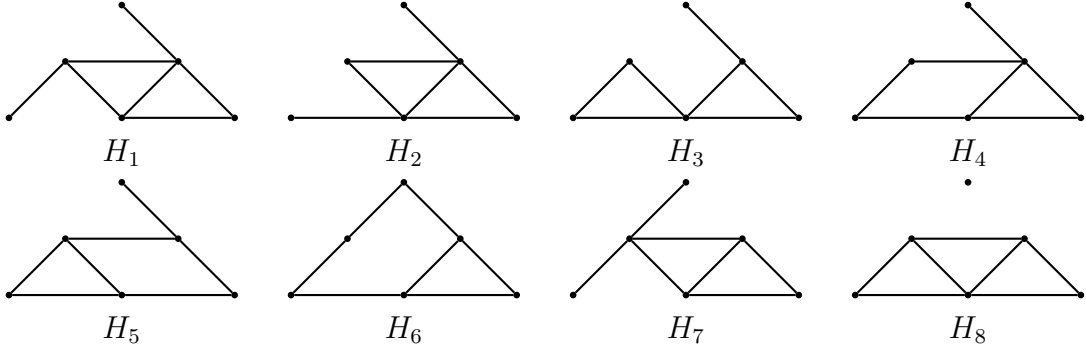

We will prove that the first $6$ of these graphs are hard to construct quickly, although it turns out that there is no need to analyze $H_1$, $H_2$, or $H_7$. The last subgraph $H_8$ is more difficult to handle, and we will bound copies of it using a different analysis.

\begin{prop} \label{prop:1-6}
For all $H_i$ such that $1 \le i \le 6$, $t(H_i,p,b^2)\lesssim b^2\ell^2$.
\end{prop}

\begin{proof}
For each graph $H_i$ with $1 \le i \le 6$, it is possible to remove two vertices of degree at least two to arrive at the path $P_3$ on four vertices. Thus, we may apply the first part of Lemma~\ref{lem:reduction} twice to show that
\[
t(H_i, p, b^2) \lesssim \ell^2 t(P_3, p, b^2),
\]
for all $1\le i \le 6$. Lastly, note that $t(P_3,p,b^2) \lesssim bt(K_2,p,b^2) \lesssim b^2$ by applying (\ref{eq:t-rec-3}) from Lemma~\ref{lemma:LB2}.

Hence, for $H_i$ with $1 \le i \le 6$, $t(H_i,p,b^2) \lesssim \ell^2 t(P_3,p,b^2) \lesssim \ell^2b^2$ as desired. \end{proof}

We must now deal with $H_8$, on which the reductions of Lemma~\ref{lem:reduction} and Lemma~\ref{lemma:LB2} are not sufficient to provide the bounds that we want.

We need one last definition. Given a graph $H$ with a distinguished vertex $u$, let $t^u(H,p,N)$ be the maximum expected number of copies of $H$ we can build in time $N$ so that $u$ maps to the same vertex in each copy. It is important to emphasize that the image of $u$ is not determined ahead of time, and we may pick it adaptively based on the queries made so far.

\begin{lem}
\label{diamond}
Let $D$ be the \emph{diamond} graph depicted below:

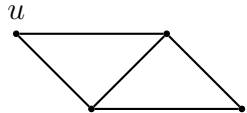
\begin{center}
\begin{tikzpicture} [rotate = 0]
\GraphInit[vstyle=Classic]
\SetVertexNoLabel
\SetGraphUnit{1}
\tikzset{VertexStyle/.append style={fill=black, minimum size=2pt, inner sep=0pt}}

\begin{scope}
\tikzset{VertexStyle/.append style={label=above:{$u$}}}
\Vertex{a}
\end{scope}
\SOEA(a){b}
\NOEA(b){c}
\SOEA(c){d}
\Edges(c,a,b,c,d,b)
\end{tikzpicture}
\end{center}
Then $t^u(D,p,b^2)\lesssim \ell^3$.
\end{lem}

\begin{proof}
As usual, we fix a query strategy maximizing $t^u(D,p,b^2)$ and let $G\subset G(2b^2,p)$ be the final graph built. Consider the following $3$ subgraphs of $D$ which we call $D_1$, $D_2$, and $D_3$ respectively.

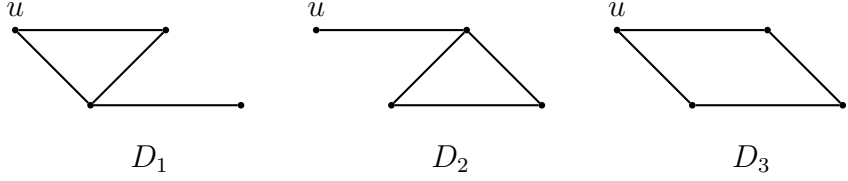
\begin{center}
\begin{tikzpicture} [rotate = 0]
\GraphInit[vstyle=Classic]
\SetVertexNoLabel
\SetGraphUnit{1}
\tikzset{VertexStyle/.append style={fill=black, minimum size=2pt, inner sep=0pt}}
\begin{scope}
\tikzset{VertexStyle/.append style={label=above:{$u$}}}
\Vertex{a}
\end{scope}
\begin{scope}
\tikzset{VertexStyle/.append style={label={[label distance=0.5cm]315:{$D_1$}}}}
\SOEA(a){b}
\end{scope}
\NOEA(b){c}
\SOEA(c){d}
\Edges(c,a,b,c)
\Edges(d,b)
\begin{scope}
\tikzset{VertexStyle/.append style={label=above:{$u$}}}
\NOEA(d){a1}
\end{scope}
\begin{scope}
\tikzset{VertexStyle/.append style={label={[label distance=0.5cm]315:{$D_2$}}}}
\SOEA(a1){b1}
\end{scope}
\NOEA(b1){c1}
\SOEA(c1){d1}
\Edges(c1,a1)
\Edges(b1,c1,d1,b1)
\begin{scope}
\tikzset{VertexStyle/.append style={label=above:{$u$}}}
\NOEA(d1){a2}
\end{scope}
\begin{scope}
\tikzset{VertexStyle/.append style={label={[label distance=0.5cm]315:{$D_3$}}}}
\SOEA(a2){b2}
\end{scope}
\NOEA(b2){c2}
\SOEA(c2){d2}
\Edges(c2,a2,b2)
\Edges(c2,d2,b2)
\end{tikzpicture}
\end{center}
Every copy of $D$ must arise from adding an edge to a graph isomorphic to one of the $D_i$. For each $u'\in V(G)$, let $X_i(u')$ be the random variable counting the number of copies of $D$ so that $u$ maps to $u'$, and the last edge built in $D$ is the edge missing from $D_i$. For each $1 \le i \le 3$, the number of copies of $D$ we can build so that $u$ maps to the same vertex in each copy, and so that $D$ arises from some copy of $D_i$, is bounded by $\max_{u'}X_i(u')$. Thus,  
\begin{equation}\label{eq:t-u}
t^u(D,p,b^2) \le \sum_{i=1}^3\mathbb{E}[\max_{u'} X_i(u')].
\end{equation}

Also, define the random variable $X_i(u', j)$ to be the number of copies of $D_i$ with $u$ mapping to $u'$ that turn into a copy of $D$ after query $j$. In particular this number is $0$ if the query $j$ finds a non-edge. We have that $X_i(u') = \sum_j X_i (u', j)$.

By the first part of Lemma~\ref{lem:poisson}, in the random graph $G(2b^2,p)$ any two vertices have $O(\ell)$ common neighbors with overwhelmingly high probability. We can assume this is the case here as the contribution to the expectation $t^u(D,p,b^2)$ from other cases is $o(1)$. In particular, this means that each new edge built can turn at most $O(\ell^2)$ copies of $D_1$, $D_2$ or $D_3$ into $D$. For example, if an edge $(u', v')$ is built in $G$, then the number of copies of $D_2$ that can be completed into $D$ is exactly the number of ways to choose a common neighbor $w'$ of $u'$ and $v'$, and then a common neighbor of $v'$ and $w'$. As we assumed that codegrees are all $O(\ell)$, there are only $O(\ell^2)$ total choices for this copy of $D_2$.

This means we may assume that $X_i(u', j)$ is stochastically dominated (up to a constant) by $\ell^2 \Bin(1,p)$. As the results of all queries are independent, it follows that $X_i(u')$ is stochastically dominated by a constant times $\ell^2 \Bin(b,p)$. Now it is a short computation that
\[
\mathbb{P}[\Bin(b,p) > 100\ell] < p^5.
\]
In particular, there exists a $C>0$ such that $\mathbb{P}[X_i(u') > C\ell^3] < p^5$ for all $1\le i \le 3$ and all $u'\in V(G)$. Also, the maximum possible number of diamonds with a given vertex $u'$ is $(2b)^3$, so
\[
t^u(D,p,b^2) \le 3C\ell^3 + \mathbb{P}[X_i(u') > C\ell^3\text{ for some }i, u']\cdot (2b)^3 = O(\ell^3),
\]
by the union bound over all $O(b^2)$ choices of $1\le i\le 3$ and $u'\in V(G)$, as desired.
\end{proof}

Finally, we deal with the graph $H_8.$ This graph behaves differently from the other ones, in that it is not the case that $t(H_8, p, b^2) = O(b^{2+o(1)})$ (in fact one can build $\Omega(b^3)$ copies of $H_8$ due to the isolated vertex). We will need to add one edge to $H_8$ and analyze the resulting graph instead.

\begin{prop}\label{prop:hstar}
Let $H^*$ be the following graph:

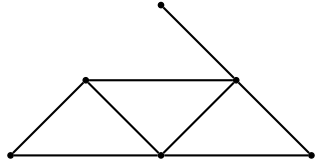
\begin{center}
\begin{tikzpicture} [rotate = -90]
\GraphInit[vstyle=Classic]
\SetVertexNoLabel
\SetGraphUnit{1}
\tikzset{VertexStyle/.append style={fill=black, minimum size=2pt, inner sep=0pt}}
\Vertex{a}
\SOEA(a){b}
\NOEA(a){c}
\SOEA(b){d}
\SOEA(c){e}
\NOEA(c){f}
\Edges(b,c,a)  
\Edges(e,d,b,e,c,f,e)
\end{tikzpicture}
\end{center}
Then $t(H^*,p,b) = O(b\ell ^3).$

\end{prop}
\begin{proof}
Let $u$ be the vertex adjacent to the leaf of $H^{*}.$  For any vertex $v$ of our final graph $G\subset G(2b^2,p)$, let $f(v)$ be the number of copies of the diamond $D$ so that $u$ maps to $v$.  As before, we may assume that all degrees are $O(b)$ as the contribution to $t(H^{*},p,b^2)$ is trivial otherwise.  Furthermore, we may assume that any two vertices have $O(\ell)$ common neighbors, as again the contribution is trivial from other cases by the first part of Lemma~\ref{lem:poisson}.  Given any $v$, there are $f(v)$ choices for a copy of $D$ including it, $d(v)$ choices for the leaf off of $v$, and $O(\ell)$ choices for the remaining vertex of $H^*$, as it is the common neighbor of $v$ and of its degree $4$ neighbor in $H^{*}.$ Thus, we obtain:
\begin{align*}
t(H^{*},p,b^2) & \lesssim \ell \mathbb{E}\left[\sum_{v}f(v)d(v)\right] \lesssim \ell \mathbb{E}\left[\left( \max_{v}f(v) \right) \left(\sum_{v}d(v)\right)\right] \\
& \lesssim b\ell \mathbb{E}\left[ \max_{v}f(v) \right] \lesssim b\ell t^u(D,p,b^2) \lesssim b\ell^4,
\end{align*}
where the last inequality follows from Lemma~\ref{diamond}. \end{proof}

Putting all of the bounds together completes the proof of Theorem~\ref{thm:lower}.

\begin{proof}[Proof of Theorem~\ref{thm:lower}]

We will apply (\ref{eq:t-rec-2}) twice to the triforce $H$. Applying it once, we find
\begin{equation}\label{eq:h-triforce}
t(H,p,b^2) \lesssim p \max\limits_{e\in E(H)} t(H\backslash e, p, b^2),
\end{equation}
and there are only two nonisomorphic subgraphs of $H$ of the form $H\setminus e$. One of them is $H^*$, for which we have $t(H^*,p,b^2) = O(b\ell^4)$ by Proposition~\ref{prop:hstar}. 

If $H'$ is the other subgraph of the form $H\setminus e$, where an inner edge is deleted from the triforce, then we apply (\ref{eq:t-rec-2}) again to find
\begin{equation}\label{eq:h'}
t(H', p, b^2) \lesssim p\max\limits_{3\le i \le 6} t(H_i, p, b^2),
\end{equation}
since all the graphs $H'\setminus e$ are isomorphic to one of $H_3, H_4, H_5, H_6$. We have by Proposition~\ref{prop:1-6} that $t(H_i, p, b^2) = O(b^2\ell^3)$.

It follows from (\ref{eq:h'}) that $t(H', p, b^2) = O(b\ell^3)$. Together with the fact that $t(H^*,p,b^2) = O(b\ell^4)$ and (\ref{eq:h-triforce}), this proves that $t(H,p,b^2) = O(\ell^4)$. The theorem follows by one application of Lemma~\ref{lemma:LB1} with $N=b^2$.\end{proof}

\section{Cliques in $G(n,1/2)$} \label{sec:clique}

In this section, we prove Theorem~\ref{thm:cliques} using an idea of Huy Pham~\cite{p}. The argument is a modification of the proof of Theorem 1 in \cite{cfgh} when the number of vertices in $G(n,p)$ is bounded beforehand.

Let $G=G(n,\frac{1}{2})$. For each vertex subset $U\in V(G)$, let $e_{t}(U)$
be the number of queries made between pairs of vertices in $U$ after query $t$. We will
study the weight function
\[
w(U,t)=\begin{cases}
2^{-\binom{|U|}{2}+e_{t}(U)} & \text{if all queries so far in \ensuremath{U} succeeded}\\
0 & \text{otherwise.}
\end{cases}
\]

In other words, $w(U,t)$ is exactly the probability that $G[U]$ is a clique conditional on the information revealed after query $t$. The standard method of conditional
expectation proceeds by studying the evolution of the function
\[
w_{k}(t)\coloneqq\sum_{|U|=k}w(U,t),
\]
which is a martingale, and has the property that a $k$-clique is
found after query $t$ only if $w_{k}(t)\ge1$. Our modification studies
instead a restricted version of this sum. Namely, define $m(U,t)$
to be the size of the maximum matching in the known edges of $U$ after query
$t$. Then,
\[
w_{k,m}(t)\coloneqq\sum_{|U|=k,m(U,t)\ge m}w(U,t).
\]

Restricting to only sets with large maximum matchings has the function
of radically reducing the number of terms in the sum $w_{k,m}(t)$.
We pay for it in that $w_{k,m}(t)$ is no longer a martingale and
its expectation is harder to study. Nevertheless, it remains true that if a $k$-clique is
found after query $t$ then $w_{k,m}(t)\ge1$ for every $m\le k/2$.
\begin{lem}
\label{lem:recursive}For any $0\le m\le k/2$ and any fixed querying
strategy that uses $t\le\binom{n}{2}$ queries,
\[
\E[w_{k,m}(t)]\le t2^{-(2k-2m-1)}\cdot\E[w_{k-2,m-1}(t)].
\]
\end{lem}

\begin{proof}
For every set $U\in\binom{[n]}{k}$, we say that $U$ is $m$-critical
at query $s$ if $s$ is the smallest number for which $m(U,s)\ge m$.
In particular, $U$ does not contribute to $w_{k,m}(t)$ until $t=s$,
after which it contributes $w(U,t)$, which is a martingale. This
means that if 
\[
w_{k,m}^{*}(s)\coloneqq\sideset{}{'}\sum_{|U|=k}w(U,s),
\]
where the sum is restricted to only sets $U$ which are $m$-critical at query $s$, then
\[
\E[w_{k,m}(t)-w_{k,m}(t-1)]=\E[w_{k,m}^{*}(t)],
\]
and so
\begin{equation}
\E[w_{k,m}(t)]=\sum_{s\le t}\E[w_{k,m}^{*}(s)].\label{eq:sum}
\end{equation}

Next, we will show
\begin{equation}
w_{k,m}^{*}(s)\le2^{-(2k-2m-2)}w_{k-2,m-1}(s).\label{eq:recursive}
\end{equation}
To see this, note that every $U$ that appears on the left side
must contain the edge $(u,v)$ built after query $s$, since $m(U,s)>m(U,s-1)$.
Furthermore, $U'=U\backslash\{u,v\}$ is a set with $k-2$ vertices
and an $m-1$ matching. Finally, every edge in $U$ but not $U'$
is incident to $(u,v)$. It is easy to check that if $(u,v)$ is an
edge that lies in every $m$-matching of $U$, then at most $2m-2$
other edges are incident to $(u,v)$. Thus, there are at least $2(k-2)-(2m-2)=2k-2m-2$
unqueried pairs in $U$ but not in $U'$, and
\[
w(U,t)\le2^{-(2k-2m-2)}w(U',t).
\]
Summing over all $m$-critical sets $U$, we get the desired inequality
(\ref{eq:recursive}). Taking expectations of both sides,
\[
\E[w_{k,m}^{*}(s)]\le2^{-(2k-2m-1)}\E[w_{k-2,m-1}(s)].
\]
Note that we gained another factor of $\frac{1}{2}$ here because
there is a $\frac{1}{2}$ chance that the query $(u,v)$ fails and
$w_{k,m}^{*}(s)=0$. Plugging this into (\ref{eq:sum}), we get
\[
\E[w_{k,m}(t)]\le2^{-(2k-2m-1)}\sum_{s\le t}\E[w_{k-2,m-1}(s)].
\]
The expectations on the right side are nondecreasing as a function
of $s$, so we can bound this by
\[
\E[w_{k,m}(t)]\le2^{-(2k-2m-1)}\sum_{s\le t}\E[w_{k-2,m-1}(s)]\le t2^{-(2k-2m-1)}\cdot\E[w_{k-2,m-1}(t)]
\]
as desired.
\end{proof}
Now we may iterate Lemma \ref{lem:recursive} until $m=0$ to prove
the following general bound.
\begin{lem}
\label{lem:cliques}For any $0\le m\le k/2$ and any fixed querying strategy
that uses $t\le\binom{n}{2}$ queries,
\[
\E[w_{k,m}(t)]\le t^{m}n^{k-2m}2^{-\binom{k}{2}+m(m-1)}.
\]
\end{lem}

\begin{proof}
We induct on $m$. The base case $m=0$ is just the unrestricted weight
function
\[
\E[w_{k,0}(t)]=\E[w_{k}(t)]=w_{k}(0)=\binom{n}{k}2^{-\binom{k}{2}}\le n^{k}2^{-\binom{k}{2}},
\]
for all $k$, as desired. Assuming the statement is true for some
$m\ge0$ and all $k\ge2m$, Lemma \ref{lem:recursive} provides the
inductive step for $m+1$ and all $k\ge2m+2$.
\end{proof}

It remains to prove Theorem~\ref{thm:cliques} using Lemma~\ref{lem:cliques}.
\begin{proof}[Proof of Theorem~\ref{thm:cliques}.]
Recall that $\E[w_{k,m}(t)]$
is an upper bound on the probability one can find a $k$-clique in
$t$ queries. By Lemma \ref{lem:cliques}, we see that whenever $n,k,t$
are such that there exists $m\le\frac{k}{2}$ for which
\[
\E[w_{k,m}(t)]\le t^{m}n^{k-2m}2^{-\binom{k}{2}+m(m-1)}<\frac{1}{2},
\]
then it is impossible to find a $k$-clique in $t$ queries in $G(n,\frac{1}{2})$
with probability at least $\frac{1}{2}$. It is cleaner to compute
the base-$2$ logarithm of this quantity. Taking $t=n^{\delta}$ and
$k=\alpha\lg n$ and writing $\ell=\lg n$ as a shorthand,
we get
\begin{eqnarray*}
\lg \Big(t^{m}n^{k-2m}2^{-\binom{k}{2}+m(m-1)}\Big) & = & (\alpha\ell-m(2-\delta))\ell-\binom{\alpha\ell}{2}+m(m-1)\\
 & \leq & \Big(\alpha-\frac{\alpha^{2}}{2}\Big)\ell^{2}-(2-\delta)m\ell+m^{2}+O(\ell).
\end{eqnarray*}
If $m=c\ell$ where $c\le\alpha/2$, then
\[
\Big(\alpha-\frac{\alpha^{2}}{2}\Big)\ell^{2}-(2-\delta)m\ell+m^{2}=\Big(\alpha-\frac{\alpha^{2}}{2}-(2-\delta)c+c^{2}\Big)\ell^{2}
\]
is minimized at $c=\frac{2-\delta}{2}$. Assuming that $\alpha\ge2-\delta$,
we find that for this choice of $c$,
\[
\lg (\E[w_{k,m}(t)])\le\Big(\alpha-\frac{\alpha^{2}}{2}-\frac{(2-\delta)^{2}}{4}\Big)\ell^{2}+O(\ell).
\]
In particular, this shows that whenever $\alpha\ge2-\delta$ satisfies
\[
\alpha-\frac{\alpha^{2}}{2}-\frac{(2-\delta)^{2}}{4}<0,
\]
then for sufficiently large $n$, it is impossible to find a clique with $\alpha\lg n$ vertices in $n^{\delta}$ queries. Thus, $\alpha_\star(\delta)$ is bounded
above by the (larger) solution to the above quadratic, which is
\[
\alpha_{+}=1+\sqrt{1-\frac{(2-\delta)^{2}}{2}}>2-\delta,
\]
as desired.
\end{proof}
\section{Concluding Remarks}\label{sec:closing}

The immediate question that arises from our work is to classify the graphs $H$ for which $f(H,p)=b^{d-o(1)}$. The natural first step is the case $d=2$. To this end, we first establish a large family of $2$-degenerate graphs $H$ for which $f(H,p) = O(b^{2-\varepsilon})$ for some $\varepsilon>0$.

\begin{Def}
We call a graph $H$ {\it $(1,1)$-degenerate} if $H$ can be vertex-partitioned into induced subgraphs $T_1, ..., T_n$ which are trees, such that for all $k\in \{1,...,n\}$ and all $v\in T_k$, 
\[
|N(v)\cap \bigcup_{i=1}^{k-1} T_i|\leq 1. 
\]
\end{Def}

It is easy to see that if $H$ is $(1,1)$-degenerate, then $H$ is $2$-degenerate. One can show by induction on the number of trees $n$ that if $H$ is $(1,1)$-degenerate, then $f(H,p)=O(b^{2-\varepsilon})$ for some $\varepsilon>0$. We prove this, and conjecture that the converse is true.

\begin{thm}
If $H$ is $(1,1)$-degenerate, then $f(H,p)=O(b^{2-\varepsilon})$ for some $\varepsilon=\varepsilon(H)>0$.
\end{thm}

\begin{proof}
We induct on the number of trees $n$ in the $(1,1)$-degenerate partition $T_1,\ldots, T_n$ of $H$. The explicit value of $\varepsilon(H)$ we pick is $\varepsilon(H)\coloneqq\max\{|V(T_1)|,\ldots, |V(T_n)|\}^{-1}$.

If $n=1$, then $H$ is a tree so that $f(H,p)=O(b)$, as desired. Now, say $n \ge 2$ and let $V(H)=V(H') \cup V(T_n)$.  For $\varepsilon' = \varepsilon(H')$, we can build a copy of $H'$ with probability at least $3/4$ in some time $O(b^{2-\varepsilon'})$ by the inductive hypothesis.  Let $\varepsilon = \min\{\varepsilon', |V(T_n)|^{-1}\}$. Now, for each $v \in V(T_n)$ there exists at most one $u \in V(H')$ so that $v$ and $u$ are adjacent. If there does not exists such a $u$, we let $C_v$ be a set of $b^{1-\varepsilon}$ previously unexplored vertices, and if there is, let $C_v$ be a set of $b^{1-\varepsilon}$ neighbors of $u$.  We can find such a set $C_v$ in at most $O(b^{2-\varepsilon})$ queries (with constant probability). 

Now, query all edges between $C_v$ and $C_{v'}$ for every pair of vertices $v, v' \in V(T_n)$.  This takes time $O(b^{2-2\varepsilon})$.  We claim that there will exist a copy of $T_n$ with $v \in C_v$ for all $v \in V(T_n)$, because the expected number of copies of $T_n$ is on the order of $b^{(1-\varepsilon)|V(T_n)|}b^{|V(T_n)|-1}=\Theta(1)$, and it is a standard result that the threshold for containment of a tree is the same as the expectation threshold. This gives that we have at least one such copy of $T_n$ with probability at least $3/4$. This copy of $T_n$ extends the original copy of $H'$ to our desired copy of $H$, all with probability at least $1/2$ in time $O(b^{2-\varepsilon})$.\end{proof}

\begin{conj}
If $H$ is a $2$-degenerate graph that is not $(1,1)$-degenerate, then $f(H,p)=b^{2-o(1)}$.
\end{conj}

In the case $d=2$, we were able to construct a particular $2$-degenerate graph $H$ for which $f(H,p) \ge b^{d-o(1)}$. The existence of such graphs when $d\ge 3$ remains open.

\begin{conj}
For all integers $d\geq 2$, there exists a $d$-degenerate graph $H$ for which $f(H,p) = b^{d-o(1)}$.
\end{conj}

There is a natural random process for constructing $d$-degenerate graphs on $n$ vertices. Namely, starting with a $K_d$, $n-d$ vertices are added one at a time, and each new vertex is given $d$ neighbors uniformly at random among the previous ones. If $n$ is sufficiently large, it is plausible that the random $d$-degenerate graph constructed in this manner should satisfy $f(H,p)=b^{d-o(1)}$ asymptotically almost surely.

\vspace{3mm}
\noindent {\bf Acknowledgments.} We would like to thank Jacob Fox, Huy Pham, and Yuval Wigderson for helpful discussions. We would also like to thank an anonymous reviewer.

\end{document}